\documentclass[12pt]{article}
\usepackage[latin1]{inputenc}
\usepackage{amsmath}
\usepackage{amsfonts}
\usepackage{amssymb,amsthm, stmaryrd}
\usepackage{graphicx}
\usepackage{xcolor}
\usepackage{multicol}
\usepackage{soul}
\usepackage[pagewise, displaymath, mathlines]{lineno}

\newcommand*\patchAmsMathEnvironmentForLineno[1]{%
  \expandafter\let\csname old#1\expandafter\endcsname\csname #1\endcsname
  \expandafter\let\csname oldend#1\expandafter\endcsname\csname end#1\endcsname
  \renewenvironment{#1}%
     {\linenomath\csname old#1\endcsname}%
     {\csname oldend#1\endcsname\endlinenomath}}%
\newcommand*\patchBothAmsMathEnvironmentsForLineno[1]{%
  \patchAmsMathEnvironmentForLineno{#1}%
  \patchAmsMathEnvironmentForLineno{#1*}}%
\AtBeginDocument{%
\patchBothAmsMathEnvironmentsForLineno{equation}%
\patchBothAmsMathEnvironmentsForLineno{align}%
\patchBothAmsMathEnvironmentsForLineno{flalign}%
\patchBothAmsMathEnvironmentsForLineno{alignat}%
\patchBothAmsMathEnvironmentsForLineno{gather}%
\patchBothAmsMathEnvironmentsForLineno{multline}%
}

\newtheorem{theorem}{Theorem}[section]

\newtheorem{question}[theorem]{Question}

\newtheorem{example}[theorem]{Example}

\newcommand{\Z}{\mbox{${\mathbb Z}$}}

\title{Rainbow considerations around the\\ Hales-Jewett theorem}

\begin{document}

\date{}

\maketitle

\begin{center}
Amanda Montejano\\[1ex]
{\small UMDI, Facultad de Ciencias\\
UNAM Juriquilla\\
Quer\'etaro, M\'exico\\
amandamontejano@ciencias.unam.mx}\\[4ex]
\end{center}

\begin{abstract}
For positive integers $t$ and $n$ let $C_t^n$ be the $n$-cube over $t$ elements, that is, the set of ordered $n$-tuples over the alphabet $\{0,\dots, t-1\}$. We address the question of whether a balanced finite coloring of $C_t^n$ guarantees the presence of a rainbow geometric or combinatorial line. For every even $t\geq 4$ and every $n$, we provide a  $\left(\frac{t}{2}\right)^n$--coloring of $C_t^n$ such that all color classes have the same size, and  without rainbow combinatorial or geometric lines. 

\end{abstract}

\section{Results}\label{sec:main}

In contrast to the classical results in Ramsey theory that seek for monochromatic structures, the realm of anti-Ramsey theory is dedicated to investigating the existence of rainbow (multicolored or totally multicolored) structures. To guarantee the existence of a structure with a color pattern that is not monochromatic, usually  we need to assume that all colors are somehow well represented. According to this, in 2003, Jungi{\'c}, Licht (Fox), Mahdian, Ne{\v{s}}etril and Radoi{\v{c}}i{\'c} \cite{jetal}  initiated what they referred to as the \emph{rainbow Ramsey theory}.  This theory focuses on the existence of rainbow structures under specific density conditions within the color classes. Recent attention to this approach is clearly reflected in various studies (see, for instance, \cite{af, cjr, dlmor, hm, jnr, ms} and references therein).

The Hales-Jewett theorem is an important result that, as stated by Graham, Rothschild, and Spencer \cite{grs}, ``reveals the hart of Ramsey theory and provides a focal point from which many results can be derived, acting  as a cornerstone for much of the more advanced work". In this paper, we present two results related to the impossibility of having an anti-Ramsey version of the Hales-Jewett theorem, thus addressing a question posed by the authors in  \cite{jetal}.

To proceed, let us establish some terminology. For integers $a$ and $b$ we use the notation $[a,b]$ to denote the set of integers $x$ such that $a\leq x\leq b$. Given non-negative integers, $t$ and $n$, we define the \emph{$n$-cube over $t$ elements}, denoted as $C_t^n$,  as the set of ordered $n$-tuples of elements taken from the alphabet $[0, t-1]$, that is, $$C_t^n=\{(x_1,x_2,...,x_n):x_i\in [0,t-1]\}.$$
The elements of $C_t^n$ are called \emph{points} or \emph{words} and, to simplify notation, we remove commas and parentheses. We also use  bold notation to refer to a point as $\textbf{x}=\,x_{1}\,x_{2}\, ...\, x_{n}\,\in C_t^n$.

A \emph{geometric line} within the $n$-cube $C_t^n$ is a sequence of $t$ distinct points,  $$\textbf{x}_0\, \textbf{x}_1\,...\,\textbf{x}_{t-1},$$ where each $\textbf{x}_i=\,x_{i\,1}\,x_{i\,2}\, ...\, x_{i\,n}\,$, that satisfy specific conditions for each coordinate $j \in[1, n]$. These conditions are as follows: either \begin{equation}x_{0\,j}=x_{1\,j}= ... =x_{t-1\,j}\end{equation} or \begin{equation}x_{i\,j}=i \mbox{ for every } 0\leq i\leq t-1\end{equation} or \begin{equation}x_{i\,j}=n-i \mbox{ for every } 0\leq i\leq t-1.\end{equation}  A geometric line that does not satisfy condition (3) is is referred to as a  \emph{combinatorial line}.
 
A \emph{$k$-coloring} of $C_t^n$ is a mapping $f:C_t^n \to [0,k-1]$ or, equivalently, a partition $C_t^n=C_0\cup C_1\cup... \cup C_{k-1}$, where each set $C_s:=f^{-1}(s)$ is called a \emph{color class}. A subset of $C_t^n$ is called \emph{monochromatic} if all its elements belong to the same color class.  A subset of $C_t^n$ is called \emph{rainbow} if all its elements belong to a different color class. 

The Hales-Jewett theorem asserts the existence of a monochromatic geometric (or combinatorial) line within any finite coloring of $C_t^n$ provided that $n$ is sufficiently large in terms of $k$ and $t$.

\begin{theorem}[Hales--Jewett, \cite{hj}] 
For every pair of positive integers $t$ and $k$ there exists a least positive integer
$HJ(t, k)$ such that, for $n\geq HJ(k,t)$, every $k$-coloring of $C_t^n$ contains a monochromatic geometric (combinatorial) line.
\end{theorem}

It is well known that the Hales-Jewett theorem implies van der Waerden's theorem on arithmetic progressions.
In 2003, Jungi{\'c} et al. \cite{jetal} explore a rainbow counterpart of van der Waerden's theorem, demonstrating the need for all color classes to be sufficiently large to guarantee a rainbow $3$-term arithmetic progression. In the same paper, the authors proposed the study of a rainbow counterpart of the  Hales-Jewett theorem as an exciting possibility,  posing the following question:

\begin{question}[Jungic \emph{et al.} \cite{jetal}]\label{q:1}
Is it true that for every balanced $t$-coloring of $C_t^n$ there exists a rainbow geometric line? 
\end{question}

 A coloring with all color classes having the same cardinality is referred to as a \emph{balanced coloring}.

As a partial negative answer of Question \ref{q:1}, the authors in \cite{jetal} provided an example of a balanced $3$-coloring of $C_3^3$ without rainbow geometric lines. The coloring is the following:

\begin{center}
$C_0=\{000,002,020,200,220,022,202,222,001\}$

$C_1=\{011,021,101,201,111,221,010,210,012\}$

$C_2=\{100,110,120,121,211,102,112,122,212\}$
\end{center}

In this paper we answer  Question \ref{q:1} for all values of $n\geq 1$ and $t\geq 2$. The case $n=1$ is trivial. For $t=2$, every balanced $2$-coloring of $C_2^n$ contains a rainbow line (a positive answer to Question \ref{q:1}).  However, this is not the case in general.  We show in Theorem \ref{thm:Q1} that, with only one exception, in the rest of the  cases the answer of Question \ref{q:1}  is negative.

\begin{theorem}\label{thm:Q1}
For every $t\geq 3$ and every $n\geq 2$, there are balanced $t$-colorings of $C_t^n$ without rainbow geometric lines, except for the case $(t,n)=(3,2)$.
\end{theorem}

The result  above may not be surprising, as there is a mapping from $C_t^n$ to $[1, t^n]$ such that every combinatorial or geometric line in $C_t^n$ corresponds to a $t$-term arithmetic progression in $[1, t^n]$, and there are constructions of balanced $t$-colorings of $[1, tN]$ without rainbow $t$-term arithmetic progressions for for all $t \geq 4$ and infinitely many values of $N$  (see \cite{af} and \cite{cjr}). In contrast,  for $t=3$, it is known that every balanced $3$-colorings of $[1,3n]$ contains a rainbow $3$-term arithmetic progression \cite{af}, which makes the case  of $C_3^n$ in Theorem \ref{thm:Q1} particulary interesting. Moreover, an intriguing problem emerges when we focus solely on balanced colorings, allowing for a greater number of colors than the intended length of a rainbow arithmetic progression. In other words, for a given $t\geq 3$, we seek to determine the minimum number of colors required to guarantee a rainbow $t$-term arithmetic progression in every balanced coloring of $[1,n]$, where $n$ is sufficiently large. Consequently, the following question naturally arises:

\begin{question}\label{q:2}
Does there exist a function $f(t)$, such that for all sufficiently large $n$ every balanced $f(t)$-coloring of $C_t^n$  contains a rainbow geometric line? 
\end{question}

The next result makes impossible a positive answer to Question \ref{q:2}.

\begin{theorem}\label{thm:main}
For every even $t\geq 4$ and every $n$, there are balanced  $\left(\frac{t}{2}\right)^n$-colorings of $C_t^n$ without rainbow geometric  lines.
\end{theorem}

We provide the proofs for Theorems \ref{thm:Q1} and  \ref{thm:main} in Section \ref{sec:proofs}. After that, in Section~\ref{sec:future}, we point and discuss potential problems to consider  for further research.

\section{Proofs}\label{sec:proofs}

From here we consider only colorings that use all possible colors. Thus, colorings are surjective mappings,  and color classes are nonempty sets. Recall that, a geometric (or combinatorial) line $\textbf{x}_0\, \textbf{x}_1\,...\,\textbf{x}_{t-1}$ of $C_t^n$ is called \emph{rainbow} with respect to the coloring $f:C_t^n \to [0,k-1]$, if   $f(\textbf{x}_i)\neq f(\textbf{x}_j)$ whenever $i\neq j$. A coloring of $C_t^n$  with no rainbow geometric (nor combinatorial) lines will be called a \emph{rainbow-free coloring}. 

For $n=2$ and  $t\geq 4$ it is easy to construct balanced $t$-colorings of $C_t^2$ without rainbow geometric lines. In fact, asking for a  balanced  rainbow-free $t$-coloring of $C_t^2$ is the same as asking for an ``anti-Latin square". That is, a $t\times t$ array of $t$ symbols, each symbol appearing $t$ times, such that in each row, each column and the two main diagonals (left and right) we have less than $t$ symbols. 

\begin{example}\label{ex:als}
For every  $t\geq 4$, define recursively a balanced rainbow-free $t$-colorings of $C_t^2$ as follows.
Start with any $4\times 4$ anti-Latin square, for instance: 
\begin{table}[h]
\centering
\label{my-label}
\begin{tabular}{|l|l|l|l|}
\hline
0 & 0  & 2  & 2    \\ \hline
0 & 0  & 2  & 2    \\ \hline
1 & 1  & 3  & 3    \\ \hline
1 & 1  & 3  & 3    \\ \hline
\end{tabular}
\end{table}

\noindent
Assume that such a $t\times t$ array, say $A_t$, exists for $t\geq 4$. Construct a $(t+1)\times (t+1)$ array from $A_t$ by adding a new left hand column and a new bottom row by using a new symbol $t+1$ times and the $t$ symbols already used in $A_t$ one more time, in such a way that the new symbol appear twice in the new column, twice in the new row, and twice in the right diagonal.  
\end{example}

By the example above we have a negative answer to  Question \ref{q:1} in the $2$-dimensional case when $t\geq 4$. Next we show that the answer still negative for every $t\geq 3$ and $n\geq 2$, except for the case $(t,n)=(3,2)$.

\begin{proof}[Proof of theorem \ref{thm:Q1}]
Since  $C_3^2$ has only $9$ elements, by a case analysis, it is not difficult to argue that every balanced $3$-coloring of $C_3^2$ contains a rainbow geometric line. 
For the remaining cases, we will construct an balanced coloring without rainbow geometric lines.  

Let $f_3$ be a balanced rainbow-free $3$-coloring of $C_3^3$ (consider, for instance, the one described at the beginning of this section). For each $t\geq 4$ let $f_t$ be a balanced rainbow-free $t$-coloring of $C_t^2$ (consider, for instance, the ones described in Example \ref{ex:als}). For $n\geq 2$, $t\geq 3$ and  $(t,n)\neq (3,2)$ define $g:C_t^n\to [0,t-1]$ as follows:  

 \begin{equation*}
g(\textbf{x})=g(x_1\,x_2\,...\,x_n)= \left\{ \begin{array}{ll}
f_3(x_1\,x_2\,x_3)& \mbox{if  $t=3$}.\\
\\
f_t(x_1\,x_2) & \mbox{if  $t\geq 4$}.\end{array} \right.
 \end{equation*}
 
 \noindent
By definition, if $t=3$, each color class  satisfies $$|g^{-1}(i)|=|f_3^{-1}(i)|(3^{n-3})=9(3^{n-3})=3^{n-1}$$ and, for $t\geq 4$,  each color class  satisfies $$|g^{-1}(i)|=|f_t^{-1}(i)|(t^{n-2})=t(t^{n-2})=t^{n-1}.$$
In both cases $g$ is a balanced $t$-coloring. In order to see that $g$ is a rainbow-free coloring consider first the case $t\geq 4$, and let  $\{\textbf{x}_0,\, \textbf{x}_1,\,...,\,\textbf{x}_{t-1}\}$ be a combinatorial line in $C_t^n$, where  $\textbf{x}_i=\,x_{i\,1}\,x_{i\,2}\, ...\, x_{i\,n}\,$. By letting $\textbf{x}'_i=x_{i\,1}\,x_{i\,2}\in C_t^2$, we have that $\{\textbf{x}'_0,\, \textbf{x}'_1,\,...,\,\textbf{x}'_{t-1}\}$ 
is either a singleton or a geometricl line in $C_t^2$. Therefore,  the set $$\{g(\textbf{x}_0)\, g(\textbf{x}_1)\,...\,g(\textbf{x}_{t-1})\}= \{f_t(\textbf{x}'_0), f_t(\textbf{x}'_1), ..., f_t(\textbf{x}'_{t-1})\}$$ is either a singleton or is the image  of a geometric line in $C_t^2$  under $f_t$, which, by hypothesis, has cardinality less than $t$ and so  $\{\textbf{x}_0,\, \textbf{x}_1,\,...,\,\textbf{x}_{t-1}\}$ is not a rainbow line. The case when $t=3$ is analogous. 
\end{proof}

Now we consider more colors than the length of the lines.

\begin{proof}[Proof of theorem \ref{thm:main}]
Let $\textbf{x} \in C_t^n$, $\textbf{x} = x_1x_2\,...\, x_n$ where each $x_i \in [0, t-1]$. We will recursively define a coloring $g_n:C_t^n \to [0,\left(\frac{t}{2}\right)^n-1]$. 
For $n=1$, let $g_1:C_t^1 \to [0, \frac{t}{2}-1]$ be defined as $$g_1(\textbf{x})=g_1(x_1)= \left\lfloor \frac{x_1}{2} \right\rfloor;$$
and for $n\geq 2$ define $$g_n(\textbf{x})=g_n(x_1x_2 \,...\, x_n)=g_{n-1}(x_1x_2 \,... \,x_{n-1})+\left(\frac{t}{2}\right)^{n-1}g_1(x_n).$$

For example,  the $3$-coloring of $C_6^1$ just defined is depicted as:

\begin{table}[h]
\centering
\label{my-label}
\begin{tabular}{|l|l|l|l|l|l|}
\hline
0 & 0  & 1  & 1  & 2 & 2  \\ \hline
\end{tabular}
\end{table} 

\begin{table}[h]
\centering
\label{my-label}
\begin{tabular}{|l|l|l|l|l|l|}
\hline
0 & 0  & 3  & 3  & 6 & 6  \\ \hline
0 & 0  & 3  & 3  & 6 & 6  \\ \hline
1 & 1  & 4  & 4  & 7 & 7  \\ \hline
1 & 1  & 4  & 4  & 7 & 7  \\ \hline
2 & 2  & 5  & 5  & 8 & 8  \\ \hline
2 & 2  & 5  & 5  & 8 & 8  \\ \hline
\end{tabular}
\end{table}

\noindent
\emph{Claim 1:} \emph{For every $n\geq 1$, $g_n$ is a balanced $\left(\frac{t}{2}\right)^n$-coloring.}

\noindent
By definition, $g_1$ is a $\left(\frac{t}{2}\right)$-coloring that uses each color twice (hence, it is balanced). We prove Claim 1 by induction on $n$. Thus, assume that $g_{n-1}$ is a balanced $\left(\frac{t}{2}\right)^{n-1}$-coloring with $|g_{n-1}^{-1}(i)|=2^{n-1}$ for every $i\in [0,\left(\frac{t}{2}\right)^{n-1}-1]$. We need to prove that $g_{n}$ is a $\left(\frac{t}{2}\right)^{n}$-coloring with $|g_{n}^{-1}(i)|=2^{n}$ for every $ i\in [0, \left(\frac{t}{2}\right)^{n}-1]$.
Consider the partitions $$\left[0,\left(\frac{t}{2}\right)^{n}-1\right]=\bigcup_{k=0}^{\frac{t}{2}-1}I_k,$$ where $I_k=[k \left(\frac{t}{2}\right)^{n-1},(k+1) \left(\frac{t}{2}\right)^{n-1}-1]$, and $$C_t^n=\bigcup_{j=0}^{t-1}\left(C_t^n\right)_j,$$ where $\left(C_t^n\right)_j=\{ x_1x_2\,...\, x_n\in C_t^n \,:\, x_n=j\}$.  
Denote by $\left.g_n\right|_j$ the coloring  $g_n$ restricted to the elements in $\left(C_t^n\right)_j$. By definition, for each $j\in[0,t-1]$, $$\left.g_n\right|_j:\left(C_t^n\right)_j\to I_{\left\lfloor\frac{j}{2}\right\rfloor},$$ and so, for each $k\in[0,\frac{t}{2}-1]$ and $i\in I_k$, we have $$g_n^{-1}(i)\subset \left(\left(C_t^n\right)_{2k}\cup \left(C_t^n\right)_{2k+1}\right).$$
Moreover, since  for every $j\in[0,t-1]$,  the coloring  $g_n$ restricted to $ \left(C_t^n\right)_j$ is isomorphic to the coloring  $g_{n-1}$ we actually have, for every $k\in[0,\frac{t}{2}-1]$ and $i\in I_k$, that
$$|g_n^{-1}(i)|=2|g_{n-1}^{-1}(i^*)|,$$ where $i^*=\left(i-k\left(\frac{t}{2}\right)^{n-1}\right)$. 
Hence, by the induction hypothesis that $g_{n-1}$ is balanced, we get 
$$|g_n^{-1}(i)|=2(2^{n-1})=2^n$$ as claimed. \\

\noindent
\emph{Claim 2:} \emph{For every $n\geq 1$, $g_n$ is a rainbow-free coloring. }

\noindent
Let  $\{\textbf{x}_0,\, \textbf{x}_1,\,...,\,\textbf{x}_{t-1}\}$ be a geometric line in $C_t^n$, where  $\textbf{x}_i=\,x_{i\,1}\,x_{i\,2}\, ...\, x_{i\,n}\,$.
In order to prove Claim 2, we will prove by induction on $n$ that $g_n(\textbf{x}_0)=g_n(\textbf{x}_1)$. This is clearly satisfied for $n=1$. 
Now,  let  $\textbf{x}'_i=\,x_{i\,1}\,x_{i\,2}\, ...\, x_{i\,n-1}\,$, thus $\textbf{x}_i=\textbf{x}'_i\, x_{i\,n}\,$. Then, $\{\textbf{x}'_0,\, \textbf{x}'_1,\,...,\,\textbf{x}'_{t-1}\}$ is either a singleton or a geometric line in $C_t^{n-1}$, and so, by the induction hypothesis, we get \begin{equation}\label{eq:ih}
g_{n-1}(\textbf{x}'_0)=g_{n-1}(\textbf{x}'_1).
\end{equation} 
Now, note that 
\begin{equation}\label{eq:g1}
g_1(x_{0\,n})= \left\lfloor \frac{x_{0\,n}}{2} \right\rfloor=\left\lfloor \frac{x_{1\,n}}{2} \right\rfloor =g_1(x_{1\,n})
\end{equation} holds true  regardless of which  condition, (1), (2) or (3), satisfy the $n$'th coordinate of the geometric line $\{\textbf{x}_0,\, \textbf{x}_1,\,...,\,\textbf{x}_{t-1}\}$.
Finally, in view of (\ref{eq:ih}) and (\ref{eq:g1}), we obtain
$$g_n(\textbf{x}_0)=g_{n-1}(\textbf{x}'_0)+\left(\frac{t}{2}\right)^{n-1}g_1(x_{0\,n})=g_{n-1}(\textbf{x}'_1)+\left(\frac{t}{2}\right)^{n-1}g_1(x_{1\,n})=g_n(\textbf{x}_1)$$ as claimed.
\end{proof}

\section{Future directions to work on}\label{sec:future}

The notion of a rainbow-free coloring coincides with the notion of
a proper strict coloring of a $\mathcal{C}$-hypergraph in the context of the
theory of coloring mixed hypergraphs (see \cite{v} or \cite{k} for some  studies in the subject). For a hypergraph $H$, the  \emph{upper chromatic number} of $H$, denoted by $\overline{\chi}(H)$, is defined as the largest integer $k$ for which there is a $k$-coloring of $V(H)$ without rainbow edges. Since   $\overline{\chi}(H)+1$ is then the minimum integer with the property that every coloring of $V(H)$ 
contains a rainbow edge, the problem of determining the upper chromatic number of a given hypergraph (or of all  members  of a given family of hypergraphs) is considered an extremal  anti-Ramsey problem. The upper chromatic number has been studied in many different contexts and has been redefined several times under different names. However, in view of the results presented in this paper, we propose the study of the balanced upper chromatic number, which is defined as the upper chromatic number restricted to balanced colorings. That is, given a hypergraph, $H$, the  \emph{balanced upper chromatic number} of $H$, denoted by $\overline{\chi}_b(H)$, is defined as the largest integer $k$ for which there is a balanced $k$-coloring of $V(H)$ without rainbow edges \cite{akm}. 
For example, when considering  the $n$-cube over $t$ elements, $C_t^n$, as a hypergraph, (with vertex set $C_t^n$, and edge set the set of geometric lines), 
Theorem \ref{thm:main} states that for every even $t\geq 4$ and every $n$, $$\overline{\chi}_b(C_t^n)\geq \left(\frac{t}{2}\right)^n.$$

The following are problems that we find interesting in this context around the Hales-Jewett theorem.

\begin{enumerate}

\item Determine or estimate the correct order of grow for 
$\overline{\chi}_b(C_t^n)$.

\item Study the behavior of the balanced upper chromatic number of $C_t^n$ viewed as an hypergraph where the set of edges as  the set of $m$-dimensional subspaces with $m$  an integer, $2\leq m\leq n$.

\item The results stated in  Section \ref{sec:main} about the existence of rainbow  arithmetic progressions in balanced colorings of $[1,n]$  can be restated as follows. Denote by $P_n^t$ 
the hypergraph with vertex set $[1,n]$ and edge set the set of $t$-term arithmetic progressions. Then, for every $n\equiv 0$ (mod $3$), $\overline{\chi}_b(P_n^3)=2$ and, for infinitely many values of $n$ and every $t\geq 4$, it happens that $\overline{\chi}_b(P_n^t)\geq t.$ Moreover,  the authors in \cite{jetal} prove that for $t\geq3$, and every $n\equiv 0$ (mod $t$), the following holds true $$\left\lfloor\frac{t^2}{4}\right\rfloor \leq \overline{\chi}_b(P_n^t)\leq \frac{t(t-1)^2}{2}-1,$$ and conjecture that $$\overline{\chi}_b(P_n^t)=\Theta (t^2).$$
 
\item Extend the study of the balanced upper chromatic number to hypergraphs where the set of vertices is $[1,n]^d$ and the set of edges is the set of homothetic copies of a given finite set  (as Gallai's theorem extend van der Waerden's theorem).  
\end{enumerate}

Another interesting  approach involves departing from balanced colorings by considering the smallest positive integer $m$ for which every $k$-coloring of $C_t^n$, where each color appears at most $m$ times, guarantees the existence of a rainbow line. This approach was started in \cite{fjr} and was considered also in \cite{ms}.

\section{Acknowledgements}

We would like to thank  Tom Brown for  many colorful and inspiring  talks around this project. This work was partially supported by DGAPA PAPIIT project number IG100822. 


\end{document}